\documentclass[10pt,a4paper,oneside]{article}
\usepackage[english]{babel}
\usepackage{a4wide,amsmath,amsthm,amssymb,url,graphicx,xspace,algorithm,algorithmic,pgf}
\usepackage[latin1]{inputenc}
\usepackage{enumitem}
\usepackage{multirow}
\usepackage{hyperref}
\usepackage{tabularx}
\usepackage{tikz}
\usepackage{allrunes}

\def\F{\mathbb{F}}

\def\N{\mathbb{N}}

\DeclareMathOperator{\PG}{PG}

\DeclareMathOperator{\AG}{AG}

\theoremstyle{definition}
\newtheorem{theorem}{Theorem}[section]
\newtheorem{lemma}[theorem]{Lemma}
\newtheorem{definition}[theorem]{Definition}
\newtheorem{remark}[theorem]{Remark}
\newtheorem{corollary}[theorem]{Corollary}

\newtheorem{example}[theorem]{Example}

\newcommand{\comments}[1]{}

\author{Maarten De Boeck\footnote{Address: UGent, Department of Mathematics, Krijgslaan 281-S22, 9000 Gent, Flanders, Belgium. \newline Email address: mdeboeck@cage.ugent.be}}
\title{The small Kakeya sets in $T^{*}_{2}(\mathcal{C})$, $\mathcal{C}$ a conic}
\date{}
\begin{document}
\maketitle

\begin{abstract}
  A Kakeya set in the linear representation $T^{*}_{2}(\mathcal{C})$, $\mathcal{C}$ a non-singular conic, is the point set covered by a set of $q+1$ lines, one through each point of $\mathcal{C}$. In this article we classify the small Kakeya sets in $T^{*}_{2}(\mathcal{C})$. The smallest Kakeya sets have size $\left\lfloor\frac{3q^{2}+2q}{4}\right\rfloor$, and all Kakeya sets with weight less than $\left\lfloor\frac{3(q^{2}-1)}{4}\right\rfloor+q$ are classified: there are approximately $\sqrt{\frac{q}{2}}$ types.
\end{abstract}

\section{Introduction}
\label{sec:introduction}

A \emph{Kakeya set} in the affine space $\AG(n,q)$ is the point set covered by a set of lines, precisely one in every direction, i.e. one affine line through every point in the hyperplane at infinity. The \emph{finite field Kakeya problem} asks for the size of the smallest Kakeya sets in $\AG(n,q)$. This problem was stated by Wolff in \cite{wol} as the finite field analogue of the classical \emph{Euclidean Kakeya problem}. We refer to \cite[Section 1.3]{tao} for a survey on the Euclidean Kakeya problem. Important results on the finite field Kakeya problem were obtained in \cite{d,dkss,ss}. In \cite{d} it was proved that a Kakeya set contains at least $\binom{q+n-1}{n}$ points.
\par The most studied finite Kakeya sets are those in $\AG(2,q)$. They arise from a set of $q+1$ lines. Not only is the minimal size of these Kakeya sets known, but there is also a classification of the small examples. We refer to \cite{bb,bdms,bm} for the results. The problem we will discuss in this article was inspired by the Kakeya problem for $\AG(2,q)$.

\par A \emph{linear representation} $T^{*}_{2}(\mathcal{K})$ is a point-line geometry which is embedded in $\AG(3,q)$, with $\mathcal{K}$ a point set in $\pi_{\infty}=\PG(2,q)$, the projective plane at infinity of $\AG(3,q)$. It is defined in the following way: the points of $T^{*}_{2}(\mathcal{K})$ are the points of $\AG(3,q)$, the lines of $T^{*}_{2}(\mathcal{K})$ are the lines of $\AG(3,q)$ whose direction is a point of $\mathcal{K}$, and the incidence is the incidence inherited from $\AG(3,q)$.
\par Since $\AG(3,q)$ together with $\pi_{\infty}=\PG(2,q)$ forms a projective geometry $\PG(3,q)$ we will have a look at some substructures of $\PG(3,q)$. A \emph{quadric} of $\PG(n,q)$ is an algebraic variety described by a homogeneous polynomial of degree $2$ in the variables $X_{0},\dots,X_{n}$. It is called \emph{singular} if a projective transformation can be found such that it can be written with less than $n+1$ variables, and \emph{non-singular} otherwise. A quadric in $\PG(2,q)$ is called a \emph{conic}. All non-singular conics are projectively equivalent. In $\PG(3,q)$ there are two types of non-singular quadrics: the \emph{hyperbolic} quadrics and the \emph{elliptic} quadrics. In this article we focus on hyperbolic quadrics. Up to a projective transformation a hyperbolic quadric is described by the equation $X_{0}X_{1}+X_{2}X_{3}=0$. On a hyperbolic quadric there are $2(q+1)$ lines, which can be divided in two groups of $q+1$ lines, called \emph{reguli}, such that two lines from the same regulus are disjoint, and two lines from different reguli meet each other. The set of lines meeting three disjoint lines $\ell_{1},\ell_{2},\ell_{3}$, is a regulus, and the set of lines meeting all lines of this regulus is also a regulus containing $\ell_{1},\ell_{2},\ell_{3}$, which is called the \emph{opposite} regulus. These two reguli form a hyperbolic quadric. More background information on quadrics and reguli can be found in \cite{hir,ht}.
\par In this article we consider the linear representation $T^{*}_{2}(\mathcal{C})$ of a non-singular conic $\mathcal{C}$ in $\pi_{\infty}$, the plane at infinity of $\AG(3,q)$. We will use the notation $T^{*}_{2}(\mathcal{C})$ throughout. We study the Kakeya sets in $T^{*}_{2}(\mathcal{C})$, which we define analogously to Kakeya sets in $\AG(2,q)$: a Kakeya line set in $T^{*}_{2}(\mathcal{C})$ is a set of lines, precisely one through each point of $\mathcal{C}$. The corresponding Kakeya set is the set of points covered by this line set. A. Blokhuis raised the question of finding the minimal size of these Kakeya sets, and to classify the small examples. The Kakeya sets in $T^{*}_{2}(\mathcal{C})$ are similar to the Kakeya sets in $\AG(2,q)$ in the sense that they arise both from a set of $q+1$ lines. 

\par In Section \ref{sec:graphs} we study this problem by looking at a related graph-theoretical problem, so we recall briefly some graph-theoretical concepts. A \emph{graph} $\Gamma$ is a pair $(V,E)$ which consists of a set $V$ of \emph{vertices} and a set $E$ of \emph{edges}, which are $2$-subsets of $V$; since we consider $E$ to be a set, graphs will be considered to be simple. If $\{u,v\}$ is an edge, then the vertices $u$ and $v$ are called \emph{adjacent}. A \emph{clique} is a set of vertices together with the edges connecting them, such that any two vertices in this set are adjacent. A clique is \emph{maximal} if it is not contained in a larger clique.
\par A graph is \emph{complete} if any two vertices are adjacent. The \emph{complete graph} on a vertex set $V$ is denoted by $K_{V}$. By $K_{r}$, $r\in\N$, we denote a complete graph on a set of $r$ vertices, without specifying the vertex set. Note that cliques are complete subgraphs. A graph is called \emph{bipartite} if there is a partition $\{V_{1},V_{2}\}$ of the vertex set $V$ 
such that two vertices in $V_{i}$ are not adjacent, $i=1,2$. It is called \emph{complete bipartite} if there is a partition $\{V_{1},V_{2}\}$ of the vertex set $V$ such that two vertices are adjacent if and only if they belong to different parts of the partition. The complete bipartite graph on the vertex sets $V'$ and $V''$ is denoted by $K_{V',V''}$. By $K_{r,s}$, $r,s\in\N$, we denote a complete bipartite graph on a set of $r$ vertices and a set of $s$ vertices, without specifying the vertex sets.

\section{The graph theory approach}
\label{sec:graphs}

In this section we will discuss a particular kind of graph. This discussion is motivated by the following definition and the two succeding lemmata.

\begin{definition}\label{graphconstruction}
  Let $\mathcal{L}$ be a Kakeya line set in $T^{*}_{2}(\mathcal{C})$. The graph $\Gamma(\mathcal{L})$ is the graph with vertex set $\mathcal{L}$ and such that two vertices are adjacent if the corresponding lines have a point in common.
\end{definition}

\begin{lemma}\label{edgedisjointmaxcliques}
  Let $\mathcal{L}$ be a Kakeya line set in $T^{*}_{2}(\mathcal{C})$. The maximal cliques of $\Gamma(\mathcal{L})$ are edge-disjoint and correspond to subsets of $\mathcal{L}$ through a common point.
\end{lemma}
\begin{proof}
  Let $\ell_{1}$, $\ell_{2}$ and $\ell_{3}$ be three lines of $\mathcal{L}$, having pairwise a point in common. The plane $\left\langle\ell_{1},\ell_{2}\right\rangle$ in $\PG(3,q)$ meets $\pi_{\infty}$ in a secant line to $\mathcal{C}$ containing $P_{1}$ and $P_{2}$, the points at infinity of $\ell_{1}$ and $\ell_{2}$. The line $\ell_{3}$ passes through a point of $\mathcal{C}$, different from $P_{1}$ and $P_{2}$, hence meets $\left\langle\ell_{1},\ell_{2}\right\rangle$ in a point. Since $\ell_{3}$ meets both $\ell_{1}$ and $\ell_{2}$, it contains the intersection point $\ell_{1}\cap\ell_{2}$. Consequently, maximal cliques of $\Gamma(\mathcal{L})$ correspond to subsets of the line set $\mathcal{L}$ through a common point.
  \par So, two different maximal cliques $C_{1}$ and $C_{2}$ of $\Gamma(\mathcal{L})$ correspond to two different subsets $\mathcal{L}_{1}$ and $\mathcal{L}_{2}$ of $\mathcal{L}$, which are two line sets, each through a common point, say $P_{1}$ and $P_{2}$. The points $P_{1}$ and $P_{2}$ are different since the maximal cliques $C_{1}$ and $C_{2}$ are different. Hence, $\mathcal{L}_{1}$ and $\mathcal{L}_{2}$ can have at most one line in common. So, $C_{1}$ and $C_{2}$ have at most a vertex in common and are thus edge-disjoint.
\end{proof}

\begin{lemma}\label{size}
  Let $\mathcal{L}$ be a Kakeya line set in $T^{*}_{2}(\mathcal{C})$, and let $k_{i}$ be the number of maximal cliques with $i$ vertices in $\Gamma(\mathcal{L})$. Then, the Kakeya set $\mathcal{K}(\mathcal{L})$ contains $q(q+1)-\sum^{q+1}_{i=1}k_{i}(i-1)$ points.
\end{lemma}
\begin{proof}
  We know that $|\mathcal{L}|=q+1$ and every line contains $q$ affine points. A point on $i$ lines, which corresponds to a maximal clique with $i$ vertices in $\Gamma(\mathcal{L})$, needs to be counted only once. So, we substract for each of these points, $i-1$ from $q(q+1)$. This proves the result
  \[
     |\mathcal{K}(\mathcal{L})|=q(q+1)-\sum^{q+1}_{i=1}k_{i}(i-1)\;.\qedhere
  \]
\end{proof}

To find small Kakeya sets, we need the value $\sum^{q+1}_{i=1}k_{i}(i-1)$ to be as large as possible, using the notation from Lemma \ref{size}. The next lemmata deal with this question.
\par In this discussion we use some Turán-style graph theoretical results. The next theorem about triangle-free graphs is due to Mantel.

\begin{theorem}[{\cite{man}}]
  A triangle-free graph on $n$ vertices contains at most $\left\lfloor\frac{n^{2}}{4}\right\rfloor$ edges. A triangle-free graph on $n$ vertices with $\left\lfloor\frac{n^{2}}{4}\right\rfloor$ edges, is the complete bipartite graph $K_{\left\lceil\frac{n}{2}\right\rceil,\left\lfloor\frac{n}{2}\right\rfloor}$.
\end{theorem}

This result was later generalised to the famous theorem by Turán on $K_{r}$-free graphs (\cite{tur}). A stability result for Mantel's theorem was proved by Hanson and Toft.

\begin{theorem}[{\cite{hto}}]
  A triangle-free graph on $n$ vertices with $\left\lfloor\frac{n^{2}}{4}\right\rfloor-l$ edges, $l<\left\lfloor\frac{n}{2}\right\rfloor-1$, is a bipartite graph.
\end{theorem}

Note that all maximal cliques of a triangle-free graph are trivially edge-disjoint. We restate the previous theorem in the following way:

\begin{corollary}\label{bijnabipartiet}
  Let $\Gamma$ be a triangle-free graph on $n$ vertices.
  \begin{enumerate}
    \item If $n$ is even, and $\Gamma$ contains $\frac{n^{2}}{4}-\varepsilon$ edges, $\varepsilon<\frac{n}{2}-1$, then $\Gamma$ arises from a complete bipartite graph $K_{\frac{n}{2}+\delta,\frac{n}{2}-\delta}$ by removing $\varepsilon-\delta^{2}$ edges, for some $\delta\leq\sqrt{\varepsilon}$, $\delta\in\N$.
    \item If $n$ is odd, and $\Gamma$ contains $\frac{n^{2}-1}{4}-\varepsilon$ edges, $\varepsilon<\frac{n-1}{2}-1$, then $\Gamma$ arises from a complete bipartite graph $K_{\frac{n+1}{2}+\delta,\frac{n-1}{2}-\delta}$ by removing $\varepsilon-\delta^{2}-\delta$ edges, for some $\delta\leq\sqrt{\varepsilon+\frac{1}{4}}-\frac{1}{2}$, $\delta\in\N$.
  \end{enumerate}
\end{corollary}

We present the main lemma of this section, which generalises the previous result.

\begin{lemma}\label{mainlemma}
  Let $\Gamma$ be a graph on $n$ vertices with edge-disjoint maximal cliques, and let $k_{i}$ be the number of maximal cliques with $i$ vertices in $\Gamma$.
  \begin{enumerate}
    \item If $n$ is even, and $\sum^{n}_{i=1}k_{i}(i-1)>\frac{n^{2}}{4}-\frac{n}{2}+1$, then $\Gamma$ arises from a complete bipartite graph $K_{\frac{n}{2}+\delta,\frac{n}{2}-\delta}$ by removing $\varepsilon-\delta^{2}$ edges, for some $\delta\leq\sqrt{\varepsilon}$, $\delta\in\N$, with $\varepsilon=\frac{n^{2}}{4}-\sum^{n}_{i=1}k_{i}(i-1)$.
    \item If $n$ is odd, and $\sum^{n}_{i=1}k_{i}(i-1)>\frac{n^{2}-1}{4}-\frac{n-1}{2}+1$, then $\Gamma$ arises from a complete bipartite graph $K_{\frac{n+1}{2}+\delta,\frac{n-1}{2}-\delta}$ by removing $\varepsilon-\delta^{2}-\delta$ edges, for some $\delta\leq\sqrt{\varepsilon+\frac{1}{4}}-\frac{1}{2}$, $\delta\in\N$, with $\varepsilon=\frac{n^{2}-1}{4}-\sum^{n}_{i=1}k_{i}(i-1)$.
  \end{enumerate}
  In both cases $\Gamma$ is a bipartite graph.
\end{lemma}
\begin{proof}
  For a graph $G$ on $n$ vertices, we denote the value $\sum^{n}_{i=1}k^{G}_{i}(i-1)$ by $C(G)$, with $k^{G}_{i}$ the number of maximal cliques with $i$ vertices in $G$.
  \par Denote the vertex set of $\Gamma$ by $V$. Let $C_{1},\dots,C_{m}$ be the maximal cliques of $\Gamma$ with at least three vertices, and let $V_{i}$ be the vertex set of $C_{i}$. We know that $|V_{i}\cap V_{j}|\leq 1$ if $i\neq j$. Also, if $x$ and $y$ are adjacent vertices in $\Gamma$, then there is a maximal clique $C_{i}$ such that $x,y\in V_{i}$.
  \par For each $i=1,\dots,m$, we choose a partition of $V_i$ with two subsets $V'_i$ and $V"_i$ such that the sizes $|V'_{i}|$ and $|V''_{i}|$ differ at most one. Let $\overline{\Gamma}$ be the graph on the vertex set $V$ with the edges of the complete bipartite subgraphs $K_{V'_{i},V''_{i}}$ and the edges of $\Gamma$ that are not contained in any of the maximal cliques $C_{i}$. In other words, the graph $\overline{\Gamma}$ is obtained by replacing the edges of the maximal cliques $C_{1},\dots,C_{m}$ (complete subgraphs $K_{V_{i}}$) in $\Gamma$ by the edges of the complete bipartite subgraphs $K_{V'_{i},V''_{i}}$, $i=1,\dots,m$.
  \par We note that $\overline{\Gamma}$ arises from $\Gamma$ by deleting $\sum_{s\geq1}k_{2s}s(s-1)+\sum_{s\geq1}k_{2s+1}s^{2}$ edges. We denote the set of edges that are in $\Gamma$ but not in $\overline{\Gamma}$ by $\mathcal{E}$. Moreover, we know that
  \[
    C(\overline{\Gamma})-C(\Gamma)=\sum_{s\geq1}k_{2s}(s-1)^{2}+\sum_{s\geq1}k_{2s+1}s(s-1)\geq0\;.
  \]
  Consequently,
  \[
    C(\overline{\Gamma})\geq C(\Gamma)>\left\lfloor\frac{n^{2}}{4}\right\rfloor-\left\lfloor\frac{n}{2}\right\rfloor+1\;.
  \]
  \par The graph $\overline{\Gamma}$ is triangle-free since any triangle of $\Gamma$ is contained in a unique maximal clique of size at least $3$. Hence, $C(\overline{\Gamma})$ equals the number of edges in $\overline{\Gamma}$, which we denote by $\left\lfloor\frac{n^{2}}{4}\right\rfloor-\varepsilon'$. Note that
  \[
    \varepsilon'+C(\overline{\Gamma})=\left\lfloor\frac{n^{2}}{4}\right\rfloor<C(\Gamma)+\left(\left\lfloor\frac{n}{2}\right\rfloor-1\right)\;,
  \]
  hence
  \[
    \varepsilon'<\left(\left\lfloor\frac{n}{2}\right\rfloor-1\right)-\left(C(\overline{\Gamma})-C(\Gamma)\right)\;. 
  \]
  \par We can apply Corollary \ref{bijnabipartiet} on $\overline{\Gamma}$ since it is triangle-free and we find that $\overline{\Gamma}$ is a bipartite graph, arising from the bipartite graph $K_{\overline{V}_{1},\overline{V}_{2}}$ by the removal of few edges; we denote this set of removed edges by $\mathcal{E}'$. Here, $\overline{V}_{1}$ and $\overline{V}_{2}$ form a partition of the vertex set $V$. If $n$ is even, then $|\overline{V}_{1}|=\frac{n}{2}-\delta$, $|\overline{V}_{2}|=\frac{n}{2}+\delta$, $\delta\in\N$ and the number of removed edges equals
  \[
    |\mathcal{E}'|=\varepsilon'-\delta^{2}<\frac{n}{2}-1-\delta^{2}-\left(C(\overline{\Gamma})-C(\Gamma)\right)\;.
  \]
  If $n$ is odd, then $|\overline{V}_{1}|=\frac{n+1}{2}+\delta$, $|\overline{V}_{2}|=\frac{n-1}{2}-\delta$, $\delta\in\N$ and the number of removed edges equals $|\mathcal{E}'|=\varepsilon'-\delta^{2}-\delta<\frac{n-1}{2}-1-\delta^{2}-\delta-\left(C(\overline{\Gamma})-C(\Gamma)\right)$.
  \par Let $x$ and $y$ be two vertices in $V_{i}$ such that $x,y\in\overline{V}_{j}$, $j=1,2$. We know that the edge $(x,y)$ is in $\mathcal{E}$. So, the vertices $x$ and $y$ are both contained in $V'_{i}$ or are both contained in $V''_{i}$. Consequently, $V_{i}\cap\overline{V}_{j}\subseteq V'_{i}$ or $V_{i}\cap\overline{V}_{j}\subseteq V''_{i}$, $j=1,2$. Since $V_{i}=(V_{i}\cap\overline{V}_{1})\cup(V_{i}\cap\overline{V}_{2})$ and both $V'_{i}$ and $V''_{i}$ are nonempty, the partitions $\{V_{i}\cap\overline{V}_{1},V_{i}\cap\overline{V}_{2}\}$ and $\{V'_{i},V''_{i}\}$ are equal. Without loss of generality we can assume $V_{i}\cap\overline{V}_{1}=V'_{i}$ and $V_{i}\cap\overline{V}_{2}=V''_{i}$.
  \par So, $\Gamma$ arises from the bipartite graph $K_{\overline{V}_{1},\overline{V}_{2}}$ by first deleting the edges of $\mathcal{E}'$ (necessarily connecting a vertex in $\overline{V}_{1}$ and a vertex in $\overline{V}_{2}$), and then adding the edges of $\mathcal{E}$, each of which connects two vertices of $\overline{V}_{1}$ or two vertices of $\overline{V}_{2}$.
  \par We distinguish now between two cases. First we assume $n$ is even. We consider the maximal clique $C_{i}$ of $\Gamma$. Two vertices in $V'_{i}$ cannot be adjacent to the same vertex of $\overline{V}_{2}\setminus V''_{i}$ since the maximal cliques of $\Gamma$ are edge-disjoint. If $|V_{i}|=2s$, and thus $|V'_{i}|=|V''_{i}|=s$, then $\mathcal{E}'$ should contain at least $(s-1)\left(\frac{n}{2}+\delta-s\right)$ edges. We also know that $C(\overline{\Gamma})-C(\Gamma)\geq(s-1)^{2}$ by the existence of $C_{i}$. Hence,
  \[
    (s-1)\left(\frac{n}{2}+\delta-s\right)\leq|\mathcal{E}'|<\frac{n}{2}-1-\delta^{2}-(s-1)^{2}\;.
  \]
  This inequality is equivalent to
  \[
    \left(\frac{n}{2}-1\right)(s-2)+\delta(\delta+s-1)<0\;,
  \]
  which is false since $s\geq2$, $n\geq2$ and $\delta\geq0$. If $|V_{i}|=2s+1$, $s\geq2$, and thus $|V'_{i}|=s$, $|V''_{i}|=s+1$ or $|V'_{i}|=s+1$, $|V''_{i}|=s$, then $\mathcal{E}'$ should contain at least $(s-1)\left(\frac{n}{2}+\delta-s-1\right)$ edges. In this case we know that $C(\overline{\Gamma})-C(\Gamma)\geq s(s-1)$. We find the inequality
  \[
    (s-1)\left(\frac{n}{2}+\delta-s-1\right)\leq|\mathcal{E}'|<\frac{n}{2}-1-\delta^{2}-s(s-1)\quad\Leftrightarrow\quad\left(\frac{n}{2}-1\right)(s-2)+\delta(\delta+s-1)<0\;,
  \]
  which is false. If $|V_{i}|=3$, then either $|V'_{i}|=1$, $|V''_{i}|=2$ or else $|V'_{i}|=2$, $|V''_{i}|=1$. In the first case $\mathcal{E}'$ should contain at least $\frac{n}{2}-\delta-1$ edges, in the second case at least $\frac{n}{2}+\delta-1$ edges. In general, $\mathcal{E}'$ should contain at least $\frac{n}{2}-\delta-1$ edges. We find the inequality
  \[
    \frac{n}{2}-\delta-1\leq|\mathcal{E}'|<\frac{n}{2}-1-\delta^{2}\quad\Leftrightarrow\quad\delta(\delta-1)<0\;,
  \]
  which is also false for all $\delta\in\N$. We conclude that $\Gamma$ does not contain maximal cliques of size at least three. Hence, $\Gamma$ is equal to $\overline{\Gamma}$ and the theorem follows.
  \par Secondly, we assume $n$ is odd. We proceed in the same way as in the case $n$ even. For $|V_{i}|=2s$ we find the inequality
  \begin{align*}
    (s-1)\left(\frac{n+1}{2}+\delta-s\right)\leq|\mathcal{E}'|&<\frac{n-1}{2}-1-\delta^{2}-\delta-(s-1)^{2}\\
    \Leftrightarrow\qquad\frac{n-1}{2}(s-2)+\delta(\delta+s)+1&<0\;.
  \end{align*}
  For $|V_{i}|=2s+1$, $s\geq2$, we find the inequality
  \begin{align*}
    (s-1)\left(\frac{n+1}{2}+\delta-s-1\right)\leq|\mathcal{E}'|&<\frac{n-1}{2}-1-\delta^{2}-\delta-s(s-1)\\
    \Leftrightarrow\qquad\frac{n-1}{2}(s-2)+\delta(\delta+s)+1&<0\;.
  \end{align*}
  For $|V_{i}|=3$ we again need to look at two different situations. Analogously, we find the inequality
  \[
    \frac{n-1}{2}-\delta-1\leq|\mathcal{E}'|<\frac{n-1}{2}-1-\delta^{2}-\delta\quad\Leftrightarrow\quad\delta(\delta-1)+\left(\frac{n-1}{2}-1\right)<0\;.
  \]
  In all three cases we find a contradiction.
\end{proof}

We describe a particular graph, whose existence shows that the bound in Lemma \ref{mainlemma} is sharp.

\begin{example}\label{sporadic}
   Let $n$ be odd. Consider the vertex sets $W_{1}$, $W_{2}$ and $W_{3}$, with $|W_{1}|=|W_{2}|=\frac{n-3}{2}$ and $W_{3}=\{x,y,z\}$. The sets $W'_{2}$ and $W''_{2}$ form a partition of $W_{2}$. The graph $G$ on the vertex set $W_{1}\cup W_{2}\cup W_{3}$ is defined by the following adjacencies. Any vertex of $W_{1}$ is adjacent to any vertex of $W_{2}\cup\{x\}$, the vertex $y$ is adjacent to all vertices of $W'_{2}\cup\{x\}$, the vertex $z$ is adjacent to all vertices of $W''_{2}\cup\{x\}$ and the vertices $y$ and $z$ are adjacent. This graph has $\frac{n^{2}-1}{4}-\frac{n-5}{2}$ edges. It contains one maximal clique of size $3$, namely $W_{3}$; all other maximal cliques are edges. Hence, $\sum^{n}_{i=1}k_{i}(i-1)=\frac{n^{2}-1}{4}-\frac{n-3}{2}$ with $k_{i}$ the number of maximal cliques with $i$ vertices in $G$.
\end{example}

\section{The classification result}
\label{sec:classification}

We give two examples of a small Kakeya set in $T^{*}_{2}(\mathcal{C})$.

\begin{example}\label{mainexample}
  We consider the linear representation $T^{*}_{2}(\mathcal{C})$ embedded in the projective space $\PG(3,q)$, with $\pi$ the plane containing the non-singular conic $\mathcal{C}$. We denote the points of $\mathcal{C}$ by $P_{0},\dots,P_{q}$. Let $\mathcal{Q}$ be a hyperbolic quadric in $\PG(3,q)$ meeting the plane $\pi$ in the conic $\mathcal{C}$, and let $\mathcal{R}$ and $\mathcal{R}'$ be the two reguli of $\mathcal{Q}$. So, through each point $P_{i}$ there is a unique line $\ell_{i}\in\mathcal{R}$ and a unique line $\ell'_{i}\in\mathcal{R}'$. We consider a partition of the points of $\mathcal{C}$ in two sets, without loss of generality $\{P_{0},\dots,P_{k-1}\}$ and $\{P_{k},\dots,P_{q}\}$, $0\leq k\leq q+1$. Let $\mathcal{L}$ be the Kakeya line set $\left(\cup^{k-1}_{i=0}\ell_{i}\right)\cup\left(\cup^{q}_{i=k}\ell'_{i}\right)$. Then, the corresponding Kakeya set $\mathcal{K}(\mathcal{L})$ covers
  \[
    kq+(q+1-k)(q-k)=\frac{3q^{2}+2q-1}{4}+\left(\frac{q+1}{2}-k\right)^{2}
  \]
  affine points.
  \par Note that $\mathcal{R}$ and $\mathcal{R}'$ can be interchanged, hence, it is sufficient to look at the examples with $0\leq k\leq\frac{q+1}{2}$. Also note that $\Gamma(\mathcal{L})$ is a complete bipartite graph.
\end{example}

\begin{example}\label{secondexample}
  We use the same notation as in the previous example. Now, we consider a partition of the points of $\mathcal{C}\setminus\{P_{q}\}$ in two sets, without loss of generality $\{P_{0},\dots,P_{k-1}\}$ and $\{P_{k},\dots,P_{q-1}\}$, $0\leq k\leq q$. Let $m$ be a secant line meeting $\mathcal{Q}$ in $P_{q}$ and a point on a line $\ell_{i}$, $0\leq i<k$ or a line $\ell'_{i}$, $k\leq i\leq q-1$. Let $\mathcal{L}$ be the Kakeya line set $\left(\cup^{k-1}_{i=0}\ell_{i}\right)\cup\left(\cup^{q-1}_{i=k}\ell'_{i}\right)\cup\{m\}$. Then, the corresponding Kakeya set $\mathcal{K}(\mathcal{L})$ covers
  \[
    kq+(q-k)(q-k)+(q-1)=\frac{3}{4}q^{2}+q-1+\left(\frac{q}{2}-k\right)^{2}
  \]
  affine points.
  \par Also here, we note that $\mathcal{R}$ and $\mathcal{R}'$ can be interchanged. Hence, it is sufficient to look at the examples with $0\leq k\leq\frac{q}{2}$. The graph $\Gamma(\mathcal{L})$ depends on the second intersection point of $m$ and $\mathcal{Q}$, the point $P_{q}$ being the first. If it is a point on a line $\ell_{i}$, $0\leq i<k$ but not on a line $\ell'_{i}$, $k\leq i\leq q-1$ (or vice versa), then $\Gamma(\mathcal{L})$ is a bipartite graph arising from a complete bipartite graph by removing all edges but one through a given vertex. If this second intersection point is a point $\ell_{i}\cap\ell'_{j}$, $0\leq i<k$ and $k\leq j\leq q-1$, then $\Gamma(\mathcal{L})$ is the graph constructed in Example \ref{sporadic} with $W'_{2}=W_{2}$.
\end{example}

In the proof of the main theorems we will use the following lemma which shows that a hyperbolic quadric is determined by a conic and two disjoint lines.

\begin{lemma}\label{quadricdefined}
  Let $C$ be a non-singular conic in a plane $\pi$ of $\PG(3,q)$ and let $m$ and $m'$ be two disjoint lines in $\PG(3,q)$, not in $\pi$, both meeting $C$. Then, there is a unique hyperbolic quadric containing $C$, $m$ and $m'$.
\end{lemma}
\begin{proof}
  Let $P_{1}$ be the point $m\cap C$ and $P_{2}$ be the point $m'\cap C$, and let $P_{3}$ be the intersection point of the tangent lines in $P_{1}$ and $P_{2}$ to the conic $C$, in the plane $\pi$. Let $P_{4}$ be a point $m\setminus\{P_{1}\}$ and let $P_{5}$ be a point $m\setminus\{P_{2}\}$. Now we choose a frame of $\PG(3,q)$ such that $P_{1}=(1,0,0,0)$, $P_{2}=(0,1,0,0)$, $P_{3}=(0,0,1,0)$, $P_{4}=(0,0,0,1)$ and $P_{5}=(1,1,1,1)$. The plane $\pi$ is given by $X_{3}=0$ and the non-singular conic $C$ is given by $\left(X_{3}=0\right)\cap\left(X_{0}X_{1}+dX^{2}_{2}=0\right)$ for some $d\in\F^{*}_{q}$.
  \par The equation of a quadric containing $C$ is $X_{0}X_{1}+dX^{2}_{2}+X_{3}(a_{0}X_{0}+a_{1}X_{1}+a_{2}X_{2}+a_{3}X_{3})=0$. All points on the line $m=\left\langle P_{1},P_{4}\right\rangle$ are on this quadric, hence $a_{0}=a_{3}=0$. Also all points on the line $m'=\left\langle P_{2},P_{5}\right\rangle$ are on this quadric, hence $a_{1}=-1$ and $a_{2}=-d$. Consequently, there is only one quadric containing $C$, $m$ and $m'$, namely the quadric given by the equation $X_{0}X_{1}+dX^{2}_{2}-X_{1}X_{3}-dX_{2}X_{3}=0$. This quadric is hyperbolic: its equation can be rewritten as $(X_{0}-X_{3})X_{1}+dX_{2}(X_{2}-X_{3})=0$. 
\end{proof}

The two following theorems present the main result of this article. A classification of the small Kakeya sets in $T^{*}_{2}(\mathcal{C})$ is proved.

\begin{theorem}\label{theoremodd}
  Let $\mathcal{L}$ be a Kakeya line set in $T^{*}_{2}(\mathcal{C})$, embedded in $\PG(3,q)$, $q$ odd, with $\mathcal{K}(\mathcal{L})$ its corresponding Kakeya set. If $|\mathcal{K}(\mathcal{L})|<\frac{3}{4}\left(q^{2}-1\right)+q$, then $\mathcal{K}(\mathcal{L})$ is a Kakeya set as described in Example \ref{mainexample}.
\end{theorem}
\begin{proof}
  In Definition \ref{graphconstruction} we defined the graph $\Gamma(\mathcal{L})$ corresponding to the Kakeya line set $\mathcal{L}$. Denote its number of maximal cliques with $i$ vertices by $k_{i}$. By Lemma \ref{size} and the assumption we know that
  \[
    q(q+1)-\sum^{q+1}_{i=1}k_{i}(i-1)=|\mathcal{K}(\mathcal{L})|<\frac{3}{4}\left(q^{2}-1\right)+q
  \]
  and thus
  \[
    \sum^{q+1}_{i=1}k_{i}(i-1)>\frac{(q+1)^{2}}{4}-\frac{q+1}{2}+1\;.
  \]
  Since $\Gamma(\mathcal{L})$ is a graph on $q+1$ vertices with edge-disjoint maximal cliques (by Lemma \ref{edgedisjointmaxcliques}), we can apply Lemma \ref{mainlemma}(1.). Since $q$ is odd the number of vertices of $\Gamma(\mathcal{L})$ is even.
  \par We denote $\frac{(q+1)^{2}}{4}-\sum^{q+1}_{i=1}k_{i}(i-1)$ by $\varepsilon$. Now, we know that $\Gamma(\mathcal{L})$ arises from a complete bipartite graph $K_{\frac{q+1}{2}+\delta,\frac{q+1}{2}-\delta}$ by removing $\varepsilon-\delta^{2}$ edges, for some $\delta\leq\sqrt{\varepsilon}$, $\delta\in\N$. Let $V_{1}$ be the set of vertices of the partition containing $\frac{q+1}{2}+\delta$ vertices and let $\mathcal{L}_{1}$ be the set of lines corresponding to it; let $V_{2}$ be the set of vertices of the partition containing $\frac{q+1}{2}-\delta$ vertices and let $\mathcal{L}_{2}$ be the set of lines corresponding to it. Since the number of edges removed from the complete bipartite graph equals $\varepsilon-\delta^{2}\leq\varepsilon<\frac{q+1}{2}-1$, we know there are two vertices in $V_{1}$ that are adjacent to all vertices in $V_{2}$, hence there are two lines, say $\ell_{0}$ and $\ell_{1}$, in $\mathcal{L}_{1}$ meeting all lines of $\mathcal{L}_{2}$.
  \par The lines $\ell_{0}$ and $\ell_{1}$ are disjoint. By Lemma \ref{quadricdefined} they define together with the non-singular conic $\mathcal{C}$ a unique hyperbolic quadric $\mathcal{Q}$. Let $\mathcal{R}$ be the regulus containing $\ell_{0}$ and $\ell_{1}$, and let $\mathcal{R}'$ be the opposite regulus. All lines in $\mathcal{L}_{2}$ meet the lines $\ell_{0}$ and $\ell_{1}$. Since they also contain a point of $\mathcal{C}$, they are lines on $\mathcal{Q}$, necessarily belonging to $\mathcal{R}'$.
  \par A line of $\mathcal{L}_{1}$ contains a point of $\mathcal{C}$, and hence it is a tangent line to $\mathcal{Q}$, a secant line to $\mathcal{Q}$ or a line on $\mathcal{Q}$. So, it meets zero, one or all lines of $\mathcal{L}_{2}$. If $\mathcal{L}_{1}$ contains a line meeting at most one line of $\mathcal{L}_{2}$, then there is a vertex in $V_{1}$ adjacent to at most one vertex of $V_{2}$. Consequently, $\Gamma(\mathcal{L})$ arises from the complete bipartite graph $K_{V_{1},V_{2}}$ by removing at least $\frac{q+1}{2}-\delta-1$ edges. However, we know that the number of removed edges equals $\varepsilon-\delta^{2}<\frac{q+1}{2}-\delta-1$, a contradiction. So, all lines of $\mathcal{L}_{1}$ are lines $\mathcal{Q}$ meeting all lines of $\mathcal{L}_{2}$, and therefore all lines of $\mathcal{L}_{1}$ belong to $\mathcal{R}$.
  \par We conclude that $\mathcal{K}(\mathcal{L})$ is a Kakeya set as described in Example \ref{mainexample}.
\end{proof}

\begin{theorem}\label{theoremeven}
  Let $\mathcal{L}$ be a Kakeya line set in $T^{*}_{2}(\mathcal{C})$, embedded in $\PG(3,q)$, $q$ even, with $\mathcal{K}(\mathcal{L})$ its corresponding Kakeya set. If $|\mathcal{K}(\mathcal{L})|<\frac{3}{4}q^{2}+q-1$, then $\mathcal{K}(\mathcal{L})$ is a Kakeya set as described in Example \ref{mainexample}.
\end{theorem}
\begin{proof}
  We proceed as in the proof of Theorem \ref{theoremodd}. We look at the graph $\Gamma(\mathcal{L})$ on $q+1$ vertices. By Lemma \ref{size} and the assumption we know
  \[
    q(q+1)-\sum^{q+1}_{i=1}k_{i}(i-1)=|\mathcal{K}(\mathcal{L})|<\frac{3}{4}q^{2}+q-1\;,
  \]
  hence
  \[
    \sum^{q+1}_{i=1}k_{i}(i-1)>\frac{(q+1)^{2}-1}{4}-\frac{(q+1)-1}{2}+1\;.
  \]
  We denote $\frac{(q+1)^{2}-1}{4}-\sum^{q+1}_{i=1}k_{i}(i-1)$ by $\varepsilon$. Applying Lemma \ref{mainlemma}(2.) we know now that $\Gamma(\mathcal{L})$ arises from a complete bipartite graph $K_{\frac{q}{2}+1+\delta,\frac{q}{2}-\delta}$ by removing $\varepsilon-\delta^{2}-\delta$ edges, $\delta\leq\sqrt{\varepsilon+\frac{1}{4}}-\frac{1}{2}$, $\delta\in\N$. Let $V_{1}$ be the set of vertices of the partition containing $\frac{q}{2}+1+\delta$ vertices and let $\mathcal{L}_{1}$ be the set of lines corresponding to it; let $V_{2}$ be the set of vertices of the partition containing $\frac{q}{2}-\delta$ vertices and let $\mathcal{L}_{2}$ be the set of lines corresponding to it. Since the number of edges removed from the complete bipartite graph equals $\varepsilon-\delta^{2}-\delta\leq\varepsilon<\frac{q}{2}-1$, we know also in this cases there are two vertices in $V_{1}$ that are adjacent to all vertices in $V_{2}$, hence there are two lines, say $\ell_{0}$ and $\ell_{1}$, in $\mathcal{L}_{1}$ meeting all lines of $\mathcal{L}_{2}$.
  \par Now, we introduce the hyperbolic quadric $\mathcal{Q}$ and the reguli $\mathcal{R}$ and $\mathcal{R}'$ as in the proof of Theorem \ref{theoremodd}. We know that all lines of $\mathcal{L}_{2}$ belong to $\mathcal{R}'$. We also know that the lines of $\mathcal{L}_{1}$ meet zero, one or all lines of $\mathcal{L}_{2}$.
  \par If there is a line in $\mathcal{L}_{1}$ meeting at most one line of $\mathcal{L}_{2}$, then $\Gamma(\mathcal{L})$ arises from the complete bipartite graph $K_{V_{1},V_{2}}$ by removing at least $\frac{q}{2}-\delta-1$ edges. We know however that the number of removed edges equals $\varepsilon-\delta^{2}-\delta<\frac{q}{2}-\delta-1$. So, we find a contradiction. As in the proof of Theorem \ref{theoremodd}, we conclude that $\mathcal{K}(\mathcal{L})$ is a Kakeya set as described in Example \ref{mainexample}.
\end{proof}

We look at the previous results in some more detail.

\begin{remark}\label{analysis}
  We first consider the case $q$ odd. We know that all small Kakeya sets in $T^{*}_{2}(\mathcal{C})$ are of the type described Example \ref{mainexample}. Among them the smallest example, corresponding to $k=\frac{q+1}{2}$, has $\gamma(q)=\frac{3q^{2}+2q-1}{4}$ points. The corresponding Kakeya line set has $\frac{q+1}{2}$ lines in each regulus.
  \par Theorem \ref{theoremodd} allows us in general to classify the $\left\lceil\sqrt{\frac{q-1}{2}}\right\rceil$ smallest types of Kakeya sets in $T^{*}_{2}(\mathcal{C})$. The second-smallest example has $\gamma(q)+1$ points. It arises from a Kakeya line set with $\frac{q+3}{2}$ lines in one regulus and $\frac{q-1}{2}$ lines in the opposite regulus. In general the $(m+1)$-smallest example has $\gamma(q)+m^{2}$ points and arises from a Kakeya line set with $\frac{q+1}{2}+m$ lines in one regulus and $\frac{q+1}{2}-m$ lines in the opposite regulus, $m\leq\left\lceil\sqrt{\frac{q-1}{2}}\right\rceil-1$. This corresponds to the construction in Example \ref{mainexample} with parameter $k=\frac{q+1}{2}-m$.
  \par Note that the result in Theorem \ref{theoremodd} is sharp. The Kakeya set of the type described in Example \ref{secondexample}, with parameter $k=\frac{q-1}{2}$, contains precisely $\frac{3}{4}\left(q^{2}-1\right)+q$ points.
  \par Now we consider the case $q$ even. We know that all small Kakeya sets in $T^{*}_{2}(\mathcal{C})$ are of the type described Example \ref{mainexample}. Among them the smallest example, corresponding with $k=\frac{q}{2}$, has $\gamma(q)=\frac{3q^{2}+2q}{4}=\frac{3}{4}q^{2}+\frac{1}{2}q$ points. The corresponding Kakeya line set has $\frac{q}{2}+1$ lines in one regulus and $\frac{q}{2}$ in the opposite.
  \par Theorem \ref{theoremeven} allows us in general to classify the $\left\lceil\sqrt{\frac{q}{2}-\frac{3}{4}}-\frac{1}{2}\right\rceil$ smallest types of Kakeya sets in $T^{*}_{2}(\mathcal{C})$. The second-smallest example has $\gamma(q)+2$ points. It arises from a Kakeya line set with $\frac{q}{2}+2$ lines in one regulus and $\frac{q}{2}-1$ lines in the opposite regulus. In general the $m$-smallest example has $\gamma(q)+m(m-1)$ points and arises from a Kakeya line set with $\frac{q}{2}+m$ lines in one regulus and $\frac{q}{2}-m+1$ lines in the opposite regulus, $m\leq\left\lceil\sqrt{\frac{q}{2}-\frac{3}{4}}-\frac{1}{2}\right\rceil$. This corresponds to the construction in Example \ref{mainexample} with parameter $k=\frac{q}{2}+1-m$.
  \par Note that the result in Theorem \ref{theoremeven} is sharp. The Kakeya set of the type described in Example \ref{secondexample}, with parameter $k=\frac{q}{2}$, contains precisely $\frac{3}{4}q^{2}+q-1$ points.
\end{remark}

For small $q$ values, the results from Theorem \ref{theoremeven}, Theorem \ref{theoremodd} and Remark \ref{analysis} give little information. For $q=2$ it yields no classification result, for $q=3,4$ only the smallest example has been classified. For $q\geq5$ at least the smallest and the second-smallest example have been classified. Therefore we study in the next remarks the small $q$ values, namely $q=2,3,4$. Given a Kakeya line set $\mathcal{L}$, we know the size of $\mathcal{K}(\mathcal{L})$ through the graph $\Gamma(\mathcal{L})$ by Lemma \ref{size}. Therefore, we study the different possibilities for $\Gamma(\mathcal{L})$.

\begin{remark}
  We first look at the case $q=2$. We already noted that the previous theorems do not yield a classification result in this case. In this remark however we will present a classification of all Kakeya sets in the case $q=2$. A Kakeya line set $\mathcal{L}$ contains $3$ lines, hence corresponds to a simple graphs on $3$ vertices. There are four nonisomorphic simple graphs on $3$ vertices.
  \begin{figure}[!ht]
    \centering
    \begin{tikzpicture}[scale=0.5]
      \draw [fill] (0,0) circle [radius=0.1];
      \draw [fill] (1,1.73) circle [radius=0.1];
      \draw [fill] (2,0) circle [radius=0.1];
      \node at (0,1.5) {$\textarc{f}_{1}$};
    
      \draw [thick] (6,0) -- (8,0);
      \draw [fill] (6,0) circle [radius=0.1];
      \draw [fill] (7,1.73) circle [radius=0.1];
      \draw [fill] (8,0) circle [radius=0.1];
      \node at (6,1.5) {$\textarc{f}_{2}$};
    
      \draw [thick] (12,0) -- (13,1.73);
      \draw [thick] (14,0) -- (13,1.73);
      \draw [fill] (12,0) circle [radius=0.1];
      \draw [fill] (13,1.73) circle [radius=0.1];
      \draw [fill] (14,0) circle [radius=0.1];
      \node at (12,1.5) {$\textarc{f}_{3}$};
    
      \draw [thick] (18,0) -- (20,0);
      \draw [thick] (18,0) -- (19,1.73);
      \draw [thick] (20,0) -- (19,1.73);
      \draw [fill] (18,0) circle [radius=0.1];
      \draw [fill] (19,1.73) circle [radius=0.1];
      \draw [fill] (20,0) circle [radius=0.1];
      \node at (18,1.5) {$\textarc{f}_{4}$};
    \end{tikzpicture}
    \caption{The simple graphs on 3 vertices.}
    \label{graphs3}
  \end{figure}
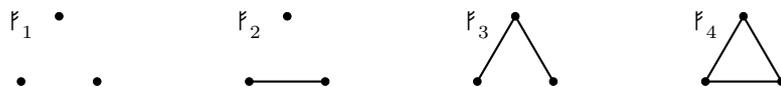
  \par In each of the graphs in Figure \ref{graphs3}, the maximal cliques are edge-disjoint. By Lemma \ref{size} the Kakeya sets corresponding to the graphs $\textarc{f}_{3}$ and $\textarc{f}_{4}$ have size $4$, the Kakeya sets corresponding to the graph $\textarc{f}_{2}$ have size $5$ and the Kakeya sets corresponding to the graph $\textarc{f}_{1}$ have size $6$. Recall that the structure of the graph gives directly a description of the Kakeya line set.
  \begin{itemize}
    \item A Kakeya line set $\mathcal{L}$ whose graph $\Gamma(\mathcal{L})$ equals $\textarc{f}_{4}$, is a cone, a set of three lines through a common affine point. There is a unique hyperbolic quadric in $\PG(3,2)$ containing the conic in the plane at infinity and two of these three lines. So, any such cone of three lines through a common affine point can also be seen as a Kakeya line set described in Example \ref{secondexample} with $k=1$ and $m$ passing through $\ell_{0}\cap\ell'_{1}$.
    \item A Kakeya line set $\mathcal{L}$ whose graph $\Gamma(\mathcal{L})$ equals $\textarc{f}_{3}$, consists of two disjoint lines, and a third line meeting both. By Lemma \ref{quadricdefined} we know that we can find a hyperbolic quadric containing the conic in the plane at infinity and the two disjoint lines. Hence, any Kakeya line set $\mathcal{L}$ whose graph $\Gamma(\mathcal{L})$ equals $\textarc{f}_{3}$ can be described by Example \ref{mainexample}, with $k=1$.
    \item A Kakeya line set $\mathcal{L}$ whose graph $\Gamma(\mathcal{L})$ equals $\textarc{f}_{2}$, consists of two intersecting lines, and a third line meeting none. Again by Lemma \ref{quadricdefined} we can see that any such Kakeya line set can be described by Example \ref{secondexample}, with $k=0$.
    \item A Kakeya line set $\mathcal{L}$ whose graph $\Gamma(\mathcal{L})$ equals $\textarc{f}_{1}$, consists of three disjoint lines. These three lines determine a unique hyperbolic quadric, which necessarily meets $\pi_{\infty}$ in the conic $\mathcal{C}$. Hence, this is a Kakeya line set described in Example \ref{mainexample}, with $k=0$.
  \end{itemize}
  We observe that all Kakeya sets for $q=2$ are described by Examples \ref{mainexample} and \ref{secondexample}.
\end{remark}

\begin{remark}
  Now, we look at the case $q=3$. Then a Kakeya line set $\mathcal{L}$ contains $4$ lines. We already classified the smallest Kakeya sets for $q=3$ in Theorem \ref{theoremodd}: they have size $8$, arise from the construction in Example \ref{mainexample} with $k=2$. In this remark we classify all Kakeya sets of size $9$. Recall that the structure of a Kakeya line set $\mathcal{L}$ is described by its graph $\Gamma(\mathcal{L})$. There are $11$ nonisomorphic simple graphs on $4$ vertices. One of them has two maximal cliques of size $3$ which have an edge in common, so cannot be a graph $\Gamma(\mathcal{L})$. The ten other graphs are presented in Figure \ref{graphs4}.
  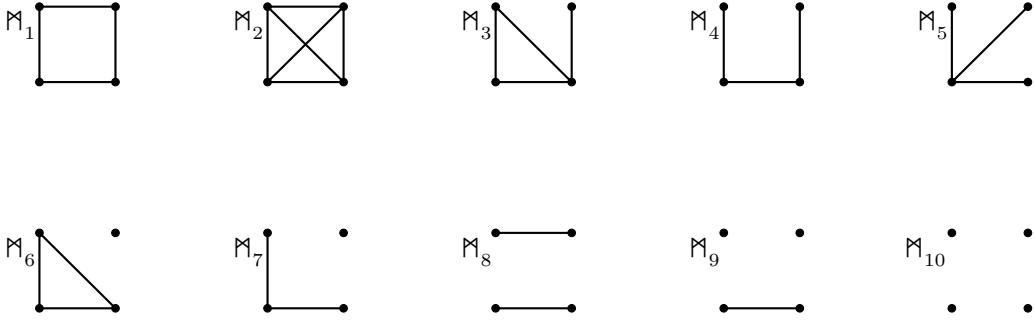
\begin{figure}[!ht]
    \centering
    \begin{tikzpicture}[scale=0.5]
      \draw [thick] (0,0) -- (2,0);
      \draw [thick] (2,0) -- (2,2);
      \draw [thick] (2,2) -- (0,2);
      \draw [thick] (0,2) -- (0,0);
      \draw [fill] (0,0) circle [radius=0.1];
      \draw [fill] (2,0) circle [radius=0.1];
      \draw [fill] (0,2) circle [radius=0.1];
      \draw [fill] (2,2) circle [radius=0.1];
      \node at (-0.5,1.5) {$\textarc{m}_{1}$};
      
      \draw [thick] (6,0) -- (8,0);
      \draw [thick] (8,0) -- (8,2);
      \draw [thick] (8,2) -- (6,2);
      \draw [thick] (6,2) -- (6,0);
      \draw [thick] (6,0) -- (8,2);
      \draw [thick] (6,2) -- (8,0);
      \draw [fill] (6,0) circle [radius=0.1];
      \draw [fill] (8,0) circle [radius=0.1];
      \draw [fill] (6,2) circle [radius=0.1];
      \draw [fill] (8,2) circle [radius=0.1];
      \node at (5.5,1.5) {$\textarc{m}_{2}$};
      
      \draw [thick] (12,0) -- (14,0);
      \draw [thick] (14,0) -- (14,2);
      \draw [thick] (12,2) -- (12,0);
      \draw [thick] (12,2) -- (14,0);
      \draw [fill] (12,0) circle [radius=0.1];
      \draw [fill] (14,0) circle [radius=0.1];
      \draw [fill] (12,2) circle [radius=0.1];
      \draw [fill] (14,2) circle [radius=0.1];
      \node at (11.5,1.5) {$\textarc{m}_{3}$};
      
      \draw [thick] (18,0) -- (20,0);
      \draw [thick] (20,0) -- (20,2);
      \draw [thick] (18,2) -- (18,0);
      \draw [fill] (18,0) circle [radius=0.1];
      \draw [fill] (20,0) circle [radius=0.1];
      \draw [fill] (18,2) circle [radius=0.1];
      \draw [fill] (20,2) circle [radius=0.1];
      \node at (17.5,1.5) {$\textarc{m}_{4}$};
      
      \draw [thick] (24,0) -- (26,0);
      \draw [thick] (24,2) -- (24,0);
      \draw [thick] (24,0) -- (26,2);
      \draw [fill] (24,0) circle [radius=0.1];
      \draw [fill] (26,0) circle [radius=0.1];
      \draw [fill] (24,2) circle [radius=0.1];
      \draw [fill] (26,2) circle [radius=0.1];
      \node at (23.5,1.5) {$\textarc{m}_{5}$};
      
      \draw [thick] (0,-6) -- (2,-6);
      \draw [thick] (0,-4) -- (0,-6);
      \draw [thick] (2,-6) -- (0,-4);
      \draw [fill] (0,-6) circle [radius=0.1];
      \draw [fill] (2,-6) circle [radius=0.1];
      \draw [fill] (0,-4) circle [radius=0.1];
      \draw [fill] (2,-4) circle [radius=0.1];
      \node at (-0.5,-4.5) {$\textarc{m}_{6}$};
      
      \draw [thick] (6,-6) -- (8,-6);
      \draw [thick] (6,-4) -- (6,-6);
      \draw [fill] (6,-6) circle [radius=0.1];
      \draw [fill] (8,-6) circle [radius=0.1];
      \draw [fill] (6,-4) circle [radius=0.1];
      \draw [fill] (8,-4) circle [radius=0.1];
      \node at (5.5,-4.5) {$\textarc{m}_{7}$};
      
      \draw [thick] (12,-6) -- (14,-6);
      \draw [thick] (14,-4) -- (12,-4);
      \draw [fill] (12,-6) circle [radius=0.1];
      \draw [fill] (14,-6) circle [radius=0.1];
      \draw [fill] (12,-4) circle [radius=0.1];
      \draw [fill] (14,-4) circle [radius=0.1];
      \node at (11.5,-4.5) {$\textarc{m}_{8}$};
      
      \draw [thick] (18,-6) -- (20,-6);
      \draw [fill] (18,-6) circle [radius=0.1];
      \draw [fill] (20,-6) circle [radius=0.1];
      \draw [fill] (18,-4) circle [radius=0.1];
      \draw [fill] (20,-4) circle [radius=0.1];
      \node at (17.5,-4.5) {$\textarc{m}_{9}$};
      
      \draw [fill] (24,-6) circle [radius=0.1];
      \draw [fill] (26,-6) circle [radius=0.1];
      \draw [fill] (24,-4) circle [radius=0.1];
      \draw [fill] (26,-4) circle [radius=0.1];
      \node at (23.3,-4.5) {$\textarc{m}_{10}$};
    \end{tikzpicture}
    \caption{The simple graphs on 4 vertices with edge-disjoint maximal cliques.}
    \label{graphs4}
  \end{figure}
  \par The graph $\textarc{m}_{1}$ corresponds to a Kakeya set of size $8$, the graphs $\textarc{m}_{i}$, $i=2,\dots,5$, correspond to a Kakeya set of size $9$, the graphs $\textarc{m}_{i}$, $i=6,7,8$, correspond to a Kakeya set of size $10$, the graph $\textarc{m}_{9}$ corresponds to a Kakeya set of size $11$, and the graph $\textarc{m}_{10}$ corresponds to a Kakeya set of size $12$.
  \par The smallest Kakeya sets for $q=3$, which have size $8$, correspond to the graph $\textarc{m}_{1}$. We look at the second-smallest Kakeya sets, those of size $9$.
  \begin{itemize}
    \item A Kakeya line set $\mathcal{L}$ whose graph $\Gamma(\mathcal{L})$ equals $\textarc{m}_{2}$, is a cone, the set of all four lines in $T^{*}_{2}(\mathcal{C})$ through a common affine point.
    \item A Kakeya line set $\mathcal{L}$ whose graph $\Gamma(\mathcal{L})$ equals $\textarc{m}_{3}$, consists of three lines through a common affine point and a line meeting one of these three lines in a different point. E.g. the Kakeya line set described in Example \ref{secondexample} with $k=1$ and $m$ passing through $\ell_{0}\cap\ell'_{2}$ yields this graph. Moreover, by Lemma \ref{quadricdefined} we can see that any Kakeya line set $\mathcal{L}$ whose graph $\Gamma(\mathcal{L})$ equals $\textarc{m}_{3}$ can be described by Example \ref{secondexample} with $k=1$ and $m$ passing through $\ell_{0}\cap\ell'_{2}$.
    \item A Kakeya line set $\mathcal{L}$ whose graph $\Gamma(\mathcal{L})$ equals $\textarc{m}_{4}$, consists of two pairs of two intersecting lines, such that one line of the first pair meets precisely one line of the other pair and the other line of the first pair does not meet a line of the second pair. Again by Lemma \ref{quadricdefined} any such Kakeya line set can be described by Example \ref{secondexample}, with $k=1$ and $m$ passing through a point of $\ell'_{1}$ (or $\ell'_{2}$) not on $\ell_{0}$.
    \item A Kakeya line set $\mathcal{L}$ whose graph $\Gamma(\mathcal{L})$ equals $\textarc{m}_{5}$, consists of three disjoint lines, and a fourth line meeting all three. All these Kakeya line sets are described by Example \ref{mainexample}, with $k=1$.
  \end{itemize}
\end{remark}

\begin{remark}
  Finally, we look at the case $q=4$. Theorem \ref{theoremeven} characterizes the Kakeya sets in $T^{*}_{2}(\mathcal{C})$ of size 14, and we characterize the Kakeya sets of size 15 here. A Kakeya line set consists of 5 lines in this case. There are 34 nonisomorphic simple graphs on 5 vertices, 9 of which have maximal cliques that are not edge-disjoint. We mentioned above the value $C(G)=\sum^{q+1}_{i=1}k^{G}_{i}(i-1)$ for a graph $G$, with $k^{G}_{i}$ the number of maximal cliques with $i$ vertices of $G$. For each of the 25 remaining graphs we can calculate this value. There is one graph with $C(G)=0$ and one graph with $C(G)=1$. There are 3 graphs with $C(G)=2$, 6 graphs with $C(G)=3$ and 10 graphs with $C(G)=4$. In Figure \ref{graphs5} we show the three graphs with $C(G)=5$ and the one graph with $C(G)=6$. These are the only graphs that can correspond to Kakeya sets of size at most $15$ by Lemma \ref{size}.
  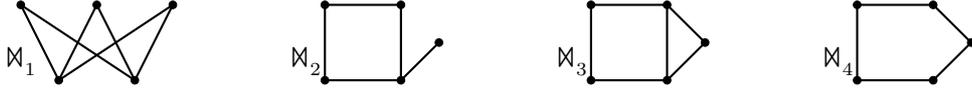
\begin{figure}[!ht]
    \centering
    \begin{tikzpicture}[scale=0.5]
      \draw [thick] (0,2) -- (1,0);
      \draw [thick] (0,2) -- (3,0);
      \draw [thick] (2,2) -- (1,0);
      \draw [thick] (2,2) -- (3,0);
      \draw [thick] (4,2) -- (1,0);
      \draw [thick] (4,2) -- (3,0);
      \draw [fill] (0,2) circle [radius=0.1];
      \draw [fill] (1,0) circle [radius=0.1];
      \draw [fill] (2,2) circle [radius=0.1];
      \draw [fill] (3,0) circle [radius=0.1];
      \draw [fill] (4,2) circle [radius=0.1];
      \node at (0,0.5) {$\textarc{d}_{1}$};
    
      \draw [thick] (8,0) -- (8,2);
      \draw [thick] (8,2) -- (10,2);
      \draw [thick] (10,2) -- (10,0);
      \draw [thick] (10,0) -- (8,0);
      \draw [thick] (10,0) -- (11,1);
      \draw [fill] (8,0) circle [radius=0.1];
      \draw [fill] (8,2) circle [radius=0.1];
      \draw [fill] (10,0) circle [radius=0.1];
      \draw [fill] (10,2) circle [radius=0.1];
      \draw [fill] (11,1) circle [radius=0.1];
      \node at (7.5,0.5) {$\textarc{d}_{2}$};
    
      \draw [thick] (15,0) -- (15,2);
      \draw [thick] (15,2) -- (17,2);
      \draw [thick] (17,2) -- (17,0);
      \draw [thick] (17,0) -- (15,0);
      \draw [thick] (17,0) -- (18,1);
      \draw [thick] (17,2) -- (18,1);
      \draw [fill] (15,0) circle [radius=0.1];
      \draw [fill] (15,2) circle [radius=0.1];
      \draw [fill] (17,0) circle [radius=0.1];
      \draw [fill] (17,2) circle [radius=0.1];
      \draw [fill] (18,1) circle [radius=0.1];
      \node at (14.5,0.5) {$\textarc{d}_{3}$};
    
      \draw [thick] (22,0) -- (22,2);
      \draw [thick] (22,2) -- (24,2);
      \draw [thick] (24,0) -- (22,0);
      \draw [thick] (24,0) -- (25,1);
      \draw [thick] (24,2) -- (25,1);
      \draw [fill] (22,0) circle [radius=0.1];
      \draw [fill] (22,2) circle [radius=0.1];
      \draw [fill] (24,0) circle [radius=0.1];
      \draw [fill] (24,2) circle [radius=0.1];
      \draw [fill] (25,1) circle [radius=0.1];
      \node at (21.5,0.5) {$\textarc{d}_{4}$};
    \end{tikzpicture}
    \caption{The simple graphs on 5 vertices that can correspond to Kakeya sets of size at most 15.}
    \label{graphs5}
  \end{figure}
  \par We already classified the smallest Kakeya sets for $q=4$: they have size $14$, arise from the construction in Example \ref{mainexample} with $k=2$, and correspond to the graph $\textarc{d}_{1}$. We look at the second-smallest Kakeya sets, those of size $15$. Recall once more that the graph describes the structure of the Kakeya line set.
  \begin{itemize}
    \item A Kakeya line set $\mathcal{L}$ whose graph $\Gamma(\mathcal{L})$ equals $\textarc{d}_{2}$ consists of five lines $m_{1},m_{2},m_{3},m'_{1},m'_{2}$ with $m_{1},m_{2},m_{3}$ pairwise disjoint, $m'_{1},m'_{2}$ disjoint, $m'_{1}$ meeting $m_{1}$, $m_{2}$ and $m_{3}$ and $m'_{2}$ meeting $m_{1}$ and $m_{2}$ but not $m_{3}$. By applying Lemma \ref{quadricdefined} we find a hyperbolic quadric containing $\mathcal{C}$ and the lines $m'_{1}$ and $m'_{2}$. It follows that $\mathcal{L}$ is a Kakeya line set that can be described by Example \ref{secondexample}, with $k=2$ and $m$ passing through a point of $\ell'_{2}$ or $\ell'_{3}$ not on $\ell_{0}$ or $\ell_{1}$ or through a point on $\ell_{0}$ or $\ell_{1}$ not on $\ell'_{2}$ or $\ell'_{3}$. Consequently, any Kakeya line set $\mathcal{L}$ whose graph $\Gamma(\mathcal{L})$ equals $\textarc{d}_{2}$ is given by Example \ref{secondexample} in the described way.
    \item It can be proved analogously, again using Lemma \ref{quadricdefined}, that any Kakeya line set $\mathcal{L}$ whose graph $\Gamma(\mathcal{L})$ equals $\textarc{d}_{3}$ can be seen as a Kakeya line set described in Example \ref{secondexample} with $k=2$ and $m$ passing through $\ell_{i}\cap\ell'_{j}$, $i=0,1$ and $j=2,3$.
    \item A Kakeya line set $\mathcal{L}$ whose graph $\Gamma(\mathcal{L})$ equals $\textarc{d}_{4}$, consists of five lines $m_{0},\dots,m_{4}$, such that $m_{i}$ meets $m_{i+1}$ and $m_{i+4}$, but not $m_{i+2}$ and $m_{i+3}$, $0\leq i\leq 4$, whereby the addition in the indices is taken modulo $5$. We will show that such a Kakeya line set cannot exist, hence the graph $\textarc{d}_{4}$ is not admissible as graph of a Kakeya line set.
    \par Let $P_{i}$ be the point $\mathcal{C}\cap m_{i}$ and let $\mathcal{Q}$ be the hyperbolic quadric defined by the disjoint lines $m_{0}$ and $m_{2}$ and the conic $\mathcal{C}$. Let $\mathcal{R}$ be the regulus of $\mathcal{Q}$ containing $m_{0}$, and let $\mathcal{R}'$ be the opposite regulus. We denote the line of $\mathcal{R}$ through $P_{i}$ by $\ell_{i}$ and the line of $\mathcal{R}'$ through $P_{i}$ by $\ell'_{i}$, $0\leq i\leq 4$. We know that $m_{0}=\ell_{0}$, $m_{1}=\ell'_{1}$ and $m_{2}=\ell_{2}$. We denote the point $\ell_{i}\cap\ell'_{j}$ by $P_{i,j}$.
    \par The lines $m_{3}$ and $m_{4}$ are both secants to the quadric $\mathcal{Q}$. Since $m_{3}$ and $m_{4}$ meet each other, $V=\left\langle m_{3},m_{4}\right\rangle$ is a plane through the line $\left\langle P_{3},P_{4}\right\rangle$. There are five planes through $\left\langle P_{3},P_{4}\right\rangle$, one of which is $\pi_{\infty}$. The four remaining planes $\pi_{1},\dots,\pi_{4}$ meet $\mathcal{Q}$ in the union of two lines $\ell_{3}\cup\ell'_{4}$, the union of two lines $\ell_{4}\cup\ell'_{3}$, the conic $\{P_{3},P_{4},P_{0,1},P_{1,2},P_{2,0}\}$ and the conic $\{P_{3},P_{4},P_{1,0},P_{2,1},P_{0,2}\}$, respectively. The line $m_{3}$ has to be a secant meeting the line $\ell_{2}$ and the line $m_{4}$ has to be a secant meeting the line $\ell_{0}$. In each of the four planes there is only one point on $\ell_{0}$ and one point on $\ell_{2}$. The line in $\pi_{1}$ through $P_{4}$ meeting $\ell_{0}$ is $\ell'_{4}$, which is not a secant line, a contradiction, so $V\neq\pi_{1}$. Analogously, starting from $P_{3}$, also $V\neq\pi_{2}$. The unique line in $\pi_{3}$ through $P_{4}$ meeting $\ell_{0}$ is $\left\langle P_{4},P_{0,1}\right\rangle$. However this line also meets $\ell'_{1}=m_{1}$, a contradiction, so $V\neq\pi_{3}$. Analogously, considering $\left\langle P_{3},P_{2,1}\right\rangle$, also $V\neq\pi_{4}$.
  \end{itemize}
  We conclude that there are only two types of Kakeya sets of size $15$ for $q=4$, both arising from the construction in Example \ref{secondexample}.
\end{remark}

\paragraph*{Acknowledgment:} The research of the author is supported by FWO-Vlaanderen (Research Foundation - Flanders) and by the BOF-UGent (Special Research Fund of Ghent University). We would like to thank Bert Seghers.

\end{document}